\documentclass[12pt]{article}

\usepackage{amsmath,amsthm,amsfonts,amssymb,bm }

\makeatletter \topmargin=0mm\headheight=0mm\headsep=0mm
\newcommand{\vp}{\varphi}
\newcommand{\Ge}{G^\varepsilon}
\newcommand{\ve}{v^\varepsilon}
\newcommand{\ue}{u^\varepsilon}
\newcommand{\be}{\begin{eqnarray}}
\newcommand{\ee}{\end{eqnarray}}
\newcommand{\ben}{\begin{eqnarray*}}
\newcommand{\enn}{\end{eqnarray*}}

\newcommand{\sgn}{{\rm sgn}\,}

\newtheorem{theorem}{\textbf Theorem}[section]
\newtheorem{lemma}{\textbf Lemma}[section]
\newtheorem{rem}{\textbf Remark}[section]

\newtheorem{prop}{\textbf Proposition}[section]

\def\endProof{{\hfill$\Box$}}

\newcounter{remark}
{\par \stepcounter{remark} {\it Remark
\arabic{section}.\arabic{remark}.}~}{\rm \endProof\par}

\def\R{\mathbb{R}}

\allowdisplaybreaks

\begin{document}
\begin{titlepage}
\title{\bf Rarefaction waves in nonlocal convection-diffusion equations }
\author{ Anna Pude{\l}ko\\
\\        
\small{AGH University of Science and Technology,}\\
\small{al. Mickiewicza 30, 30-059, Krak\'ow, Poland.}\\
\small{        (pudelko@agh.edu.pl)}}

\date{}
\end{titlepage}
\maketitle

\begin{abstract}
We consider the ``convection-diffussion'' equation 
$u_t=J*u-u-uu_x,$ where $J$ is a probability density. 
We supplement this equation with step-like initial conditions  and prove a convergence of corresponding solution towards 
a rarefaction wave, {\it i.e.}  
a unique entropy solution of the Riemann problem for the nonviscous Burgers equation.
Methods and tools used in this paper are inspired by those used
in [Karch, Miao and Xu, SIAM J. Math. Anal. {\bf 39} (2008), no. 5, 1536--1549.], where the fractal Burgers equation was studied.

\medskip


\medskip


\medskip

{\bf
\noindent
\date{\today}}

\end{abstract}
\section{Introduction}
\setcounter{equation}{0}

The goal of this work is to study asymptotic properties of solutions
to the Cauchy problem for the following nonlocal convection-diffussion equation 
\begin{equation}
\label{rownanie}
u_t={\cal L} u-uu_x, \qquad x\in \mathbb{R},~t>0,
\end{equation}
where the nonlocal operator ${\cal L}$ is defined by the formula 
\begin{equation}
\label{operator}
{\cal L} u=J*u-u, \quad \text {with} \quad J\in L^1(\mathbb R),~ J\geqslant 0,
\end{equation}
and `` * '' denotes the convolution with respect to the space variable.
We supplement this problem with the step-like initial condition satisfying 
\begin{equation}
\label{warunek}
u(x,0)=u_0(x)\to u_{\pm} \qquad \text{when}\quad x\to\pm \infty  
 \end{equation}
with some constants $u_{-}<u_{+}$. The precise meaning of this condition is given in (\ref{as u01}) and (\ref{as u02}), below.

\vskip 0,5cm
Equation (\ref{rownanie}) with the particular kernel $J(x)=\frac{1}{2}e^{-\vert x\vert}$
can be obtained from the following system modelling a radiating gas \cite{H}
\begin{equation}
\label{numer}
 u_t+uu_x+q_x=0,\quad -q_{xx}+q+u_x=0 \qquad\text{for} \quad x\in \mathbb R,~t\geqslant 0.
\end{equation}
Indeed, the second equation in (\ref{numer}) can be formally solved to obtain $q=-\tilde J u_x,$
with a kernel $J(x)=\frac{1}{2}e^{-\vert x\vert}$ that is the fundamental solution of the operator $-\frac{d^2}{dx^2}+I.$ 
Thus, substituting  $ q_x=-\tilde Ju_{xx}=u-J*u$
into first equation in (\ref{numer}) we obtain an equation which is formally equivalent to (\ref{rownanie})-(\ref{operator}).
The derivation of system (\ref{numer}) from the Euler system for a perfect compressible fluid 
coupled with an elliptic equation for the temperature can be found in \cite{KT}.

In this work, we consider more general kernels (see our assumptions (\ref{as J}), below), because
the general integral operator ${\cal L} u=J*u-u$ models long range interactions and appears in many problems ranging 
from micro-magnetism \cite{MGP, MOPT_1, MOPT_2}, neural network \cite{EM}, hydrodynamics \cite{R} to ecology
\cite{CMS, C, DK, KM, M}, and \cite{SSN}. For example, in some population dynamic models, such
an operator is used to model the dispersal of individuals through their environment \cite{F_1, F_2, HMMV}.
We also refer the reader to a series of papers \cite{AB, BFRW, Ch, Co_1, Co_2, CoD, CoDM} on travelling fronts and to \cite{CoDM_2} on pulsating fronts for the equation $u_t=J*u-u+f(x,u).$

The equation in (\ref{rownanie})-(\ref{operator}) with the particular kernel $J(x)=\frac{1}{2}e^{-\vert x\vert}$ (thus in the context of modelling 
radiating gases) with various classes of initial data have been recently intensively studied. 
For existence and uniqueness results, we refer the reader to \cite{KN} and \cite{LM}.
In \cite{CH}, Chmaj gave an answer to an open problem stated by Serre in \cite{S_2} 
concerning existence of travelling wave solutions to equation (\ref{rownanie})-(\ref{operator}) with more general kernel.
Here, we refer the reader to the recent work \cite{ChJ}, for generalizations of those results and for additional references.  

The large time behaviour of solution to equation (\ref{rownanie})-(\ref{operator}) was considered eg. in \cite{KN, S, L, KT}. In the case of initial data $u_0$ satisfying 
$u_0(x)\to u_{\pm}$ when $x\to\pm \infty,$ with $u_->u_+$, Serre \cite{S} showed the $L^1$-stability of shock profiles.  
Asymptotical stability of smooth travelling waves was proved in \cite{KN}.
For initial data $u_0\in L^1(\mathbb R)\cap L^\infty(\mathbb R),$ Lauren\c{c}ot \cite{L} showed the convergence 
of integrable and bounded weak
solutions of (\ref{rownanie})-(\ref{operator}) towards a source-type solution to the viscous Burgers equation.
Here, we recall also recent works \cite{IR, IIS-D}, where a doubled nonlocal 
version of equation (\ref{rownanie}) (namely, where the Burgers flux is replaced by nonlinear term in 
convolution form) was studied together with initial conditions from $L^1(\mathbb R)\cap L^\infty(\mathbb R).$  

The large time behaviour of solutions to problem (\ref{rownanie})-(\ref{warunek}) when  $J(x)=\frac{1}{2}e^{-\vert x\vert}$ and
$u_-<u_+$ was studied by Kawashima and Tanaka \cite{KT}, where a specific structure of this model was used to show
the convergence of solutions towards rarefaction waves, under suitable smallness conditions on initial data.

The goal of this work is to generalize the result from \cite{KT} by considering less regular initial condition with 
no smallness assumption and more general kernel $J.$ 
To deal with such a problem, we develope methods and tools, which are inspired by those used in \cite{KMX} where 
the fractal Burgers equation was studied.  


\section{Main result}
\setcounter{equation}{0}

First, we recall that the explicit function
\begin{equation}\label{rarefaction}
w^R(x,t)=\left\{
\begin{aligned}
&u_-\,,\quad &&x/t\leq u_-,\\
&x/t\,,\quad &&u_-\leq x/t\leq u_+, \\
& u_+\,,\quad &&x/t\geq u_+,
\end{aligned}
\right.
\end{equation}
is called a rarefaction wave and satisfies the following Riemann problem
\begin{eqnarray*}
&&w^R_t+w^R w^R_x=0,\label{eq-rar}\\
&&w^R(x,0)=w_0^R(x)=\left\{
\begin{aligned}
u_-\,,\quad x<0,\\
u_+\,,\quad x>0.
\end{aligned}\right.\label{ini-rar}
\end{eqnarray*}
in a weak (distributional sense). Moreover, this is the unique entropy solution. Such rarefaction waves appear as
asymptotic profiles when $t\to\infty$ of solutions to the viscous Burgers equation 
\begin{equation*}
\label{burgers}
u_t-u_{xx}+uu_x=0
\end{equation*}
supplemented with an initial datum $u(x,0)=u_0(x)$, satisfying $u_0-u_-\in L^1((-\infty,0))$ and $u_0-u_+\in L^1((0,\infty)),$ 
(cf. \cite{HN, IO} and Lemma \ref{w_R}, below).
Below, we use also the following regularized problem
\begin{eqnarray}
&&w_t-w_{xx}+w w_x=0,\label{eq-app}\\
&&w(x,0)=w_0(x)=
\left\{
\begin{aligned}
u_-\,,\quad x<0,\\
u_+\,,\quad x>0,
\end{aligned}\right.
\label{ini-app}
\end{eqnarray}
which solution is called smooth approximation of the rarefaction wave (\ref{rarefaction}).

The purpose of this paper is to show that weak solutions of the nonlocal Cauchy problem (\ref{rownanie})-(\ref{warunek})
exist for all $t\geqslant 0$ and converge as $t\to\infty$ towards the rarefaction wave.

Here, as usual, a function $u\in L^\infty(\mathbb R\times [0,\infty))$ is called a weak solution 
to problem (\ref{rownanie})-(\ref{warunek}) if for every test function $\varphi\in C^\infty_c(\mathbb R\times [0,\infty))$ we have
\begin{equation*}
\label{ue_slabe_2}
-\int_\mathbb R\int_0^\infty u\vp_t dtdx-\int_\mathbb R u_0(x)\vp(x,0)~dx= 
\int_\mathbb R\int_0^\infty u {\cal L}\vp ~dtdx+\frac{1}{2}\int_\mathbb R\int_0^\infty u^2\vp_x~dtdx.
\end{equation*}

\noindent In the following, we assume that ${\cal L}u=J*u-u$ with
\begin{equation}
\label{as J}
\begin{split}
&J,~\vert x\vert^2 J\in L^1(\mathbb R), \quad  \int_{\mathbb R} J(x) dx=1,\\
&J(x)=J(-x)\quad \text{and}\quad J(x)\geqslant 0\quad \text{for all} \quad x\in\mathbb R.
\end{split}
\end{equation}
Moreover, we consider initial conditions satisfying
\begin{equation}
\label{as u01}
u_0-u_-\in L^1((-\infty,0))\quad \text{and} \quad u_0-u_+\in L^1((0,\infty))                                         
\end{equation}
as well as
\begin{equation}
\label{as u02}
u_{0,x}\in L^1(\mathbb R) \quad \text{and}\quad u_{0,x}(x)\geqslant 0\quad \text{a.e. in} ~~\mathbb R.                                        
\end{equation}

Now, we formulate the main result of this work on the rate of convergence of solutions to problem (\ref{rownanie})-(\ref{warunek})
towards the rarefaction wave (\ref{rarefaction}).

\begin{theorem}\label{th main}
Assume that the kernel $J$ satisfies (\ref{as J}) and the initial datum $u_0$ has properties stated in (\ref{as u01}) and (\ref{as u02}).
Then, there exists a unique weak solution $u=u(x,t)$ of problem (\ref{rownanie})-(\ref{warunek}) 
with the following property:
for every $p\in [1,\infty]$ there is a constant $C>0$ such that
\begin{equation}
\label{glowna nierownosc}
\Vert u(t)-w^R(t)\Vert_p\leqslant C t^{-(1-1/p)/2}[\log(2+t)]^{(1+1/p)/2} 
\end{equation}
for all $t>0.$
\end{theorem}

\begin{rem}
Although the nonlocal operator ${\cal L}u=J*u-u$ has no regularizing property as
{\it e.g.} the Laplace operator, we have still global-in-time continuous solutions, 
because, for non-decreasing initial condition, the nonlinear term 
in equation (\ref{rownanie}) does not develope shocks in finite time. 
\end{rem}

The paper is organized as follows. In the next section, we gather results concerning an equation regularized by the usual viscosity term
and auxiliary lemmas on properties of the nonlocal operator ${\cal L}$.
The main result on the large time behaviour of solutions to the regularized problem
is shown in Section 4. The convergence of regularized solutions to a weak solution of the nonlocal problem (\ref{rownanie})-(\ref{warunek})  and
Theorem \ref{th main} are proved in Section~5.

\vskip 0,5cm
\noindent{\bf Notation.}
By $\|\cdot\|_p$ we denote the $L^p$-norm  of a function defined on $\R.$
Integrals without integration limits are defined on the whole line $\mathbb R.$
Several numerical constants are denoted by  $C$.

\section{Regularized problem}
In this section, we consider the regularized problem
\begin{equation}
\label{rownanie zregularyzowane}
u_t=\varepsilon u_{xx}+{\cal L} u-uu_x, \qquad x\in \mathbb{R},~t>0
\end{equation}
\begin{equation}
 \label{warunek ue}
u(x,0)=u_0(x).
 \end{equation}
with fixed $\varepsilon>0.$ Our first goal is to show that this initial value problem has a unique smooth global-in-time solution. 

\begin{theorem} ({\it Existence of solutions.})
\label{istnienie zregularyzowanego}
If $u_0\in L^\infty(\mathbb R)$ and ${\cal L} u=J*u-u,$ where the kernel $J$ satisfies (\ref{as J}), then 
the regularized problem (\ref{rownanie zregularyzowane})-(\ref{warunek ue}) has a solution $\ue\in L^\infty(\mathbb R\times[0,\infty]).$ 
Moreover, this solution satisfies:
\item[(i)] $u\in C^\infty(\mathbb R\times (0,\infty))$ and all its derivatives are bounded on $\mathbb R\times(t_0,\infty)$ for all $t_0 > 0,$
\item[(ii)] for all $(x,t)\in\mathbb R\times [0,\infty)$  
\begin{equation}
\label{L}
\underset{x\in\mathbb R}{{\rm essinf}~ u_0}\leqslant \ue(x,t)\leqslant \underset{x\in\mathbb R}{{\rm esssup}~u_0}
\end{equation} 
\item[(iii)] $u$ satisfies the equation (\ref{rownanie zregularyzowane}) in the classical sense,
\item[(iv)] $u(t)\to u_0$ as $t\to 0,$ in $L^\infty (\mathbb R)$ weak-* and in $L^p_{loc}(\mathbb R)$ for all $p\in [1,\infty).$
\newline This is a unique solution of problem (\ref{rownanie zregularyzowane})-(\ref{warunek ue}) in the sense of the integral formulation (\ref{duhamel}), below.
\end{theorem}

In the following theorem we collect other properties of solutions to the regularized problem.

\begin{theorem}
\label{wlasnosci rozwiazania}
Assume that the kernel $J$ satisfies (\ref{as J}). Let $\ue$ be a solution of the regularized problem corresponding to
an initial condition $u_0$ satisfying (\ref{as u01})
\newline If $ u_{0,x}\in L^1(\mathbb R)$ then  
\begin{equation}
\label{prawo zach}
\int \ue_x(x,t) ~dx=\int u_{0,x}(x)~ dx.
\end{equation}
If $u_{0,x}\geqslant 0$ then
\begin{equation*}
\ue_x(x,t)\geqslant 0
\end{equation*}
for all $x\in \mathbb R$ and $t\geqslant 0.$
\newline Moreover, for two initial conditions $u_0, ~\bar u_0$ satisfying (\ref{as u01})-(\ref{as u02}), the corresponding solutions
$\ue,~\bar u^\varepsilon$ satisfy
\begin{equation}
\label{kontrakcja}
\Vert \ue(t)-\bar u^\varepsilon(t)\Vert_1\leqslant \Vert u_0-\bar u_0\Vert_1.                       
\end{equation}
\end{theorem}
{\it Proof of Theorem \ref{istnienie zregularyzowanego}.}
Following the usual procedure, based on the Duhamel principle, we rewrite problem (\ref{rownanie zregularyzowane})-(\ref{warunek ue})
in the itegral form
\begin{equation}
\label{duhamel}
\begin{split}
 \ue (x,t)=(\Ge(\cdot,t)*\ue_0)(x)&+\int_0^t(\Ge(\cdot,t-s)*{\cal L}\ue(\cdot,s))(x)~ ds\\
&-\int_0^t (\Ge(\cdot,t-s)*\ue(\cdot,s)\ue_x(\cdot,s))(x)~ ds ,
\end{split}
\end{equation}
where $\Ge(x,t)=(4\pi\varepsilon t)^{-1/2} e^{-\frac{\vert x\vert^2}{4\varepsilon t}}$ 
is the fundamental solution of the heat equation $u_t=~\varepsilon u_{xx}.$
It is a completely standard reasoning (details can be found for example in (\cite[Section 5]{DGV}), 
based on the Banach contraction principle, that the integral equation (\ref{duhamel}) 
has a unique local-in-time regular solution on $[0,T]$ with properties stated in (i), (iii) and (iv).
Here, one should notice that the second term on the right hand side of the equation (\ref{duhamel}) does not cause 
any problem to adapt the arguments from (\cite[Section 5]{DGV}) in our case. This is due to the fact that the convolution operator
${\cal L}$ is bounded on $L^\infty(\mathbb R).$
Hence, we skip these details. This solution is global-in-time because of estimates (\ref{L}) 
which we are going to prove below.
\endProof

In the proof of the comparison principle expressed by inequalities (\ref{L}) we adapt ideas described in \cite{K_eveq}.
It is based in the following auxiliary results.

\begin{lemma}\label{lem:pass}
Let $\varphi \in C^3_b(\R).$ 
If the sequence $\{x_n\}\subset \R$ satisfies
$\varphi(x_n)\to \underset{x\in\R}{\sup}~\varphi(x)$
then 
\item[(i)] $\underset{n\to\infty}{\lim}~\varphi'(x_n)=0$
\item[(ii)] $\underset{n\to\infty}{\limsup}~\varphi''(x_n)\leqslant 0$
\item[(iii)] $\underset{n\to\infty}{\limsup}~{\cal L}\varphi(x_n)\leqslant 0.$
\end{lemma}

\begin{proof}
Since $\varphi''$ is  bounded, there exists $C>0$ such that
\begin{equation}
\label{Taylor_1}
\underset{x\in\R}{\sup}~\varphi(x)\geqslant \varphi(x_n-z)\geqslant \varphi(x_n)-\varphi'(x_n)~z-Cz^2
\end{equation}
for every $z\in \R.$ Since the sequence $\{\varphi'(x_n)\}$ is bounded, passing to the subsequence, we can assume that 
$\varphi'(x_n)\to p.$
Consequently, passing to the limit in (\ref{Taylor_1})  we obtain the inequality
\begin{equation*}
0\geqslant -pz-Cz^2 \quad\quad\mbox{for every} \quad\quad z\in \R,
\end{equation*}
which imediately implies $p=0$.

To prove inequality (ii), we use an analogous argument involving the inequality
\begin{equation}
\label{Taylor_2}
\underset{x\in\R}{\sup}~\varphi(x)\geqslant \varphi(x_n-z)\geqslant \varphi(x_n)-\varphi'(x_n)~z+\frac{1}{2} \varphi''(x_n)~z^2-Cz^3
\end{equation}
for all $z>0,$ where $C=\frac{1}{6}\Vert \varphi'''\Vert_{\infty}.$ Passing to the limit superior in (\ref{Taylor_2}), 
denoting $q=\underset{n\to\infty}{\limsup} \varphi''(x_n)$ and using (i) we obtain the inequality
\begin{equation*}
0\geqslant \frac{1}{2}qz-Cz^3 \quad\quad\mbox{for all} \quad\quad z>0.
\end{equation*}
Choosing $z>0$ arbitrarily small we deduce from this inequality that $q\leqslant 0$ which completes the proof of (ii).

Now, we prove that $\underset{n\to\infty}{\limsup}~{\cal L}\varphi(x_n)\leqslant 0.$
Note first, that by the definition of the sequence $\{x_n\}$, we have
\begin{equation*}
\varphi(x_n-z)-\varphi(x_n)\leqslant \underset{x\in\R}{\sup}~\varphi(x) -\varphi(x_n)\to 0
\quad\text{as}\quad n\to\infty. 
\end{equation*}
Hence, $\underset{n\to\infty}{\limsup} \Big(\varphi(x_n-z)-\varphi(x_n)  \Big)\leqslant 0. $
Applying the Fatou lemma to the expression
\begin{equation*}
{\cal L}\varphi(x_n)=
\int\Big(
\varphi(x_n-z)-u(x_n)\Big)J(z)~dz
\end{equation*}
ends the proof of (iii).
\end{proof}

We are now in a position to prove the comparison principle for equations with the nonlocal operator ${\cal L}.$

\begin{prop}
\label{eveq}
Assume that $u\in C_b(\R\times [0,T])\cap C^3_b(\R\times [\varepsilon ,T])$
is the solution of the equation
\begin{equation}
\label{eq:lin}
u_t=u_{xx}+{\cal L} u-b(x,t)u_x,
\end{equation}
where ${\cal L}$ is the nonlocal convolution operator given by (\ref{operator}) and 
$b=b(x,t)$ is a given and sufficiently regular real-valued function. Then
\begin{equation*}
u(x,0)\leqslant 0 \quad\mbox{implies}\quad u(x,t)\leqslant 0 \quad \mbox{for all}\quad x\in \R,~t\in [0,T].
\end{equation*}
\end{prop}

\begin{proof}
The function $\Phi(t)=\underset{x\in\R}{\sup}~u(x,t)$ is well-defined and continuous.
Our goal is to show that $\Phi$ is locally Lipschitz and $\Phi'(t)\leqslant 0$ almost everywhere.
To show the Lipschitz continuity of $\Phi$, for every $\varepsilon>0$ we choose $x_\varepsilon$ such that
\begin{equation*}
\underset{x\in\R}{\sup}~u(x,t)=u(x_\varepsilon,t)+\varepsilon.
\end{equation*}

Now, we fix $t,s\in I$, where $I\subset (0,T)$ is a bounded 
and closed interval and we suppose (without loss of generality)
that $\Phi(t)\geq \Phi(s)$.
Using the definition of $\Phi$ and regularity of $u$ we obtain
\begin{equation*}
\begin{split}
0\leq \Phi(t)-\Phi(s) &= \sup_{x\in\R} u(x,t)-\sup_{x\in\R} u(x,s)\\
&\leq \varepsilon +u(x_\varepsilon, t)-u(x_\varepsilon, s)\\
&\leq \varepsilon +\sup_{x\in\R} |u(x,t)-u(x,s)|\\
&\leq \varepsilon + |t-s| \sup_{x\in\R,t\in I}|u_t(x,t)|.
\end{split}
\end{equation*}
Since $\varepsilon>0$ and $t,s\in I$ are  arbitrary, we immediately  obtain
that the function $\Phi$ is locally Lipschitz,
hence, by the Rademacher theorem,  differentiable almost everywhere, as well. 

Let us now  differentiate $\Phi(t)=\underset{x\in\R}{\sup}~u(x,t)$ with respect to $t>0$.
By the Taylor expansion, for $0<s<t$, we have
\begin{equation*}
u(x,t)= u(x,t-s)+s u_t(x,t)+Cs^2. 
\end{equation*}

Hence, using equation (\ref{eq:lin}), we obtain
\begin{equation}
\label{es1}
u(x,t)\leqslant\underset{x\in\R}{\sup}~u(x,t-s) +s \Big(u_{xx}(x,t)+{\cal L} u(x,t)-b(x,t)u_x(x,t)\Big) +Cs^2.
\end{equation}
Substituting in (\ref{es1})  $x=x_n$, 
where $u(x_n,t)\to \underset{x\in\R}{\sup}~u(x,t)$ as $n\to\infty$,
passing to the limit using Lemma \ref{lem:pass}, we obtain the inequality
\begin{equation*}
\underset{x\in\R}{\sup}~u(x,t)\leq \underset{x\in\R}{\sup}~u(x,t-s) +Cs^2
\end{equation*}
which can be transformed into
\begin{equation*}
\frac{\Phi(t)-\Phi(t-s)}{s}\leq Cs.
\end{equation*}
For $s\searrow 0$, we obtain $\Phi'(t)\leq 0$ in those $t$, where $\Phi$ is differentiable.
\end{proof}

{\it Proof of Inequalities (\ref{L}).}
Let $m=\underset{x\in\mathbb R}{{\rm esssup}~u_0}$ then, since ${\cal L}m=0$, the function $\ve(x,t)=\ue(x,t)-m$ 
satisfies the following equation
\begin{equation*}
 \ve_t=\ve_{xx}+{\cal L}\ve-(\ve+m)\ve_x
\end{equation*}
Now, we use Proposition \ref{eveq} with $b(x,t)=\ve(x,t)+m$ to conclude that $\ve(x,t)\leqslant 0,$ so $\ue(x,t)\leqslant m$ 
for all $x\in \mathbb R,~ t\in [0,T],$ for arbitrary $T>0.$
The proof of the second inequality $\underset{x\in\mathbb R}{{\rm essinf}~u_0}\leqslant \ue(x,t)$ is completely analogous, hence we skip it.
\endProof


\vskip 0,5cm
{\it Proof of Theorem \ref{wlasnosci rozwiazania}.} 
In order to show the equality (\ref{prawo zach}), we differentiate Duhamel's formula (\ref{duhamel}), and we obtain
\begin{equation}
\label{---}
\begin{split}
 \ue_x(x,t)=(\Ge(\cdot,t)*\ue_{0,x})(x)&+\int_0^t(\Ge(\cdot,t-s)*{\cal L}\ue_x(\cdot,s))(x)~ ds\\
&-\int_0^t (\Ge_x(\cdot,t-s)*\ue(\cdot,s)\ue_x(\cdot,s))(x)~ ds.
\end{split}
\end{equation}
Then, integrating (\ref{---}) over $\mathbb R,$ we have
\begin{equation}
\label{pochodna}
\begin{split}
\int \ue_x(x,t)~dx&=\int (\Ge(\cdot,t)*\ue_{0,x})(x)~dx+ \int_0^t \int (\Ge(\cdot,t-s)*{\cal L}\ue_x(\cdot,s))(x)~dxds\\
&-\int_0^t \int (\Ge_x(\cdot,t-s)*\ue(\cdot,s)\ue_x(\cdot,s))(x))~dxds.
\end{split}
\end{equation}
Since $\int \Ge(x,t)~dx=1,$ the second term on the right hand side of (\ref{pochodna}) is egual to zero by the equality (\ref{zero}).
Now, making of use the equality $\int \Ge_x(x,t)~dx=0,$ leads to zero in the last term on the right hand side of (\ref{pochodna}), 
and that ends the proof of (\ref{prawo zach}).

To prove nonnegativity of $\ue_x,$ we first differentiate equation (\ref{rownanie zregularyzowane}) with respect to $x,$ and we have
\begin{equation}
\label{+}
(\ue_x)_t=\varepsilon (\ue_{xx})_x+{\cal L} \ue_x-\left(\ue \ue_x\right)_x, \qquad x\in \mathbb{R},~t>0.
\end{equation}
Next, we multiply (\ref{+}) by $(\ue_x)^-=\max\{-\ue_x,0\},$ and we integrate the resulting equation 
over $\mathbb R,$ to obtain
\begin{equation}
 \label{++}
\int (\ue_x)_t(\ue_x)^-~dx=\varepsilon \int (\ue_{xx})_x~(\ue_x)^-~dx+\int (\ue_x)^-~{\cal L} \ue_x~dx-\int \left(\ue \ue_x\right)_x~(\ue_x)^-~dx.
\end{equation}
Now, we notice that the integral on the left-hand side of (\ref{++}) is equal to $\frac{1}{2}\frac{d}{dt}\int_{\ue_x\leqslant 0} \left[(\ue_x)^-\right]^2~ dx.$
Straightforward calculations, based on the integration by parts in the first and third term of the right-hand side of (\ref{++}) lead to
\begin{equation*}
\label{+++}
 \frac{1}{2}\frac{d}{dt}\int_{\ue_x \leqslant 0} \left[(\ue_x)^-\right]^2 dx=-\int_{\ue_x \leqslant 0} \left[(\ue_x)^-\right]^2dx+
\int_{\ue_x \leqslant 0} (\ue_x)^-~{\cal L} \ue_xdx-\frac{1}{2} \int_{\ue_x \leqslant 0}  \left[(\ue_x)^-\right]^3dx.
\end{equation*}
By Lemmas \ref{convex} and \ref{Kato}, we have 
$\int_{\ue_x \leqslant 0} (\ue_x)^-~{\cal L} \ue_x~dx\leqslant 0.$ In a consequence, we obtain
\begin{equation*}
\label{4+}
 \frac{1}{2}\frac{d}{dt}\int_{\ue_x \leqslant 0} \left[(\ue_x)^-\right]^2 dx \leqslant 0,
\end{equation*}
and, this immediately implies 
\begin{equation*}
 \int_{\ue_x \leqslant 0} \left[(\ue_x)^-\right]^2 dx\leqslant \int_{\ue_x \leqslant 0} \left[(u_{0,x})^-\right]^2 dx.
\end{equation*}
By nonnegativity assumption from (\ref{as u02}) imposed on $u_{0,x}$ we have $(\ue_x)^-=0$ on the set $\{\ue_x \leqslant 0\},$  
thus, in a consequence, we have
$\ue_x(x,t)\geqslant 0$ for all $x\in \mathbb R$ and $t>0.$

To prove the $L^1$-contraction property in (\ref{kontrakcja}) is sufficient to repeat the reasons from Lemma \ref{log} below, hence
we do not reproduce it, here.
\endProof


\section{Convergence of regularized solutions towards rarefaction wave}

Now, we show that a solution to the regularized problem satisfies certain decay estimates and converges 
towards a rarefaction wave with all estimates independent of $\varepsilon>0.$
The main result of this section reads as follows.

\begin{theorem}
\label{jeszcze nie wiem jaki label}
Let $u=\ue(x,t)$ be the solution of regularized problem (\ref{rownanie zregularyzowane})-(\ref{warunek ue}), with the kernel 
$J$ satisfying (\ref{as J}) and the initial data $u_0$ satisying (\ref{as u01})-(\ref{as u02}),
from Theorem~\ref{istnienie zregularyzowanego}.
For every $p\in[1,\infty]$ there exists $C=C(p)>0$ independent of $t$ and of $\varepsilon>0$ such that
\begin{equation}
\label{decay zregularyzowanego}
\Vert \ue_x(t)\Vert_p\leqslant t^{-1+1/p}\Vert u_{0,x}\Vert_1^{1/p} 
\end{equation}
and
\begin{equation}
\label{log zregularyzowanego}
\Vert \ue(t)-w^R(t)\Vert_p\leqslant C t^{-(1-1/p)/2}[\log(2+t)]^{(1+1/p)/2} 
\end{equation}
for all $t>0,$ where $w^R=w^R(x,t)$ is the rarefaction wave (\ref{rarefaction}).
\end{theorem}
 
We proceed the proof of this theorem by proving preliminary inequalities 
involving the nonlocal operator ${\cal L}.$
 
\begin{lemma}\label{Kato}
For every $\varphi\in L^1(\mathbb R)$ we have ${\cal L}\varphi\in L^1(\mathbb R).$ Moreover,
\begin{equation}
\label{zero}
  \int {\cal L} \varphi~dx = 0
\end{equation}
and
\begin{equation}
\label{signum}
 \int {\cal L} \varphi~\sgn\varphi ~dx\leqslant 0.
\end{equation}
\end{lemma}

\begin{proof} The function ${\cal L}\varphi$ is integrable by the Young inequality and the following calculation
\begin{equation*}
 \Vert {\cal L} \varphi\Vert_1\leqslant \Vert\varphi\Vert_1+\Vert J*\varphi\Vert_1\leqslant \Vert\varphi\Vert_1(1+\Vert J\Vert_1).
\end{equation*}
Since $\int J(x)~dx=1$, we obtain (\ref{zero}) immediately by applying the Fubini theorem.
\newline\noindent Since ${\cal L}\varphi=J*\varphi-\varphi,$ to prove inequality (\ref{signum}), it is sufficient to use the estimates
\begin{equation*}
\left\vert \int J*\varphi\cdot\sgn \varphi~dx\right\vert\leqslant\int\int J(y)\vert\varphi(x-y)\vert~dxdy=\int \vert \varphi(x)\vert~dx
\end{equation*}
by the Fubini theorem and assumptions (\ref{as J}).
\end{proof}
 
\begin{lemma}\label{convex}
Let $\varphi\in L^1(\mathbb R)$ and $g\in C^2(\mathbb R)$ be a convex function. Then
\begin{equation}
\label{convex inequality}
{\cal L} g(\varphi)\geqslant g'(\varphi){\cal L} \varphi \qquad a.e.
\end{equation}
\end{lemma}

\begin{proof}
 The convexity of the function $g$ leads to the following inequality
\begin{equation*}
 g(\varphi(x-y))-g(\varphi(x))\geqslant g'(\varphi(x))[\varphi(x-y)-\varphi(x)].
\end{equation*}
Multiplying this inequality by $J(y)$ and integrating it with respect to $y$ over $\mathbb R$ we obtain the inequality (\ref{convex inequality}).
\end{proof}

For simplicity of the exposition,
we first formulate some auxiliary lemmas.
We start with known results concerning the initial value problem for the viscous Burgers equation (\ref{eq-app})-(\ref{ini-app}).
The following estimates can be deduced from the explicit formula for solutions to the problem (\ref{eq-app})-(\ref{ini-app}).
We refer the reader to \cite{HN} for detailed calculations, and for additional improvements to \cite{KT}.

\begin{lemma}\label{w_R}
Problem (\ref{eq-app})-(\ref{ini-app}) with  $u_-<u_+$ has the unique solution $w(x,t)$ satisfying $u_-<w(t,x)<u_+$
and  $w_x(t,x)>0$ for all $(x,t)\in\R\times(0,\infty)$.

Moreover, for every $p\in[1,\;\infty]$, there is a constant
$C=C(p,u_-,u_+)>0$ such that
\begin{equation*}
\|w_x(t)\|_{p}\leq C t^{-1+1/p},
\quad
\|w_{xx}(t)\|_{p}\leq C t^{-3/2+1/(2p)}
\end{equation*}
and
\begin{equation*}
\|w(t)-w^R(t)\|_{p}\leq C t^{-(1-1/p)/2},
\end{equation*}
for all $t>0$, where $w^R(x,t)$ is the rarefaction wave
(\ref{rarefaction}).
\end{lemma}

Our goal is to estimate $\Vert \ue(t)-w(t)\Vert_{p}$ where $\ue=\ue(x,t)$ is a solution of regularized problem 
(\ref{rownanie zregularyzowane})-(\ref{warunek ue}) and $w=w(x,t)$ is a smooth approximation of the rarefaction wave $w^R.$
First, we deal with the $L^1$-norm.

\begin{lemma}\label{log}
Assume that $\ue=\ue(x,t)$ is a solution of problem (\ref{rownanie zregularyzowane})-(\ref{warunek ue}) 
from Theorem \ref{istnienie zregularyzowanego}.
Let $w=w(x,t)$ be the smooth approximation of a rarefaction wave. Then, there exists a constant $C>0$ independent of $t$ and 
of $\varepsilon>0$ such that
\begin{equation*}
\label{log_nierownosc}
 \Vert \ue(t)-w(t)\Vert_1\leqslant C\log (2+t)\qquad \text{for all} \quad t>0.
\end{equation*}
\end{lemma} 

\begin{proof}
The function $\ve(x,t)=\ue(x,t)-w(x,t)$ satisfies the following equation
\begin{equation*}
 \ve_t-{\cal L} \ve+\left(\frac{(\ve)^2}{2}+\ve w\right)_x={\cal L} w-w_{xx}.
\end{equation*}
We multiply it by $\sgn \ve$ and we integrate over $\mathbb R$ to obtain
\begin{equation}
\label{*}
\begin{split}
  \frac{d}{dt}\int \vert \ve\vert ~dx-\int {\cal L} \ve ~\sgn \ve dx&+ \frac{1}{2}\int \left((\ve)^2+2\ve w\right)_x ~\sgn \ve dx\\
&=\int ({\cal L} w-w_{xx})~\sgn \ve dx.
\end{split}
\end{equation}

\noindent By Lemma \ref{Kato}, the second
term on the left-hand side of (\ref{*}) is non-negative. 
For the third term, we approximate the
sgn function by a smooth and nondecreasing function
$\varphi=\varphi(x)$.
Thus, we obtain
\begin{align*}
\int [(\ve)^2+2\ve w]_x \varphi(\ve) dx
&=-\int ((\ve)^2+2\ve w)\varphi'(\ve)\ve_x \;dx\\
&=-   \int \Psi(\ve)_x\; dx+
\int w_x \Phi(\ve)\; dx,
\end{align*}
where $\Psi(s)=\int_0^sz^2\varphi'(z)\,dz$ and
$\Phi(s)=\int_0^s2z\varphi'(z)\,dz$.
Here, the first term on the right hand side equals zero and
the second one is nonnegative
because  $w_x\geq 0$ and
$\Phi(s)\geq 0$ for all $s\in \R$.
Hence, an approximation argument gives
$\int[(\ve)^2+~2\ve w]_x \sgn \ve \,dx\geq 0$. 

\noindent Now, we estimate the term on the rght hans side of (\ref{*}).
First, we notice that using the Taylor formula, we have
\begin{equation*}
\begin{split}
  {\cal L} w(x,s)&=(J*w-w)(x,s)=\int J(y)[w(x-y,s)-w(x,s)]dy\\
&=\int J(y) y w_x(x,s)~dy+ \int J(y) \frac{y^2}{2} w_{xx}(x+\theta y,s)~dy, 
\end{split}
\end{equation*}
where $\int J(y) y w_x(x,s) dy=w_x(x,s)\int J(y) y dy=0$ by the symmetry assumption from (\ref{as J}).
Therefore, by assumption (\ref{as u01}), we can estimate the integral
 on the right-hand side of (\ref{*}) as follows
\begin{equation*}
\begin{split}
\left\vert \int ({\cal L} w-w_{xx})~\sgn \ve~dx\right\vert&\leqslant \int J(y) \frac{y^2}{2} \int \vert w_{xx}(x+\theta y,s)\vert~dxdy+\Vert w_{xx}\Vert_1\\
&\leqslant C \Vert w_{xx}\Vert_1.
\end{split}
\end{equation*}
Consequently, applying these estimates to inequality (\ref{*}) we obtain the following differential inequality
\begin{equation}
\label{**}
\frac{d}{dt} \Vert \ve (t)\Vert_1\leqslant C \Vert w_{xx}(t)\Vert_1.
\end{equation}
Now, by Lemma \ref{w_R}, we have the inequality $\Vert w_{xx}(t)\Vert_1\leqslant Ct^{-1}$ for all $t>0,$ which combined with
(\ref{**}) ater integration completes the proof of Lemma \ref{log}.
\end{proof}

Now, we are in a position to prove the convergence of regularized solutions towards rarefaction wave.

\bigskip
\noindent
{\it Proof of Theorem \ref{jeszcze nie wiem jaki label}} {\it Part I. Decay estimates.}

In the case $p=1$, we use the equality (\ref{prawo zach}) from Theorem \ref{wlasnosci rozwiazania}.
Since $\ue_x\geqslant 0,$ we have 
\begin{equation}
\label{-}
 \Vert \partial_x \ue(t)\Vert_1=\Vert \partial_x u_0\Vert_1 \qquad \text{for all} \quad t\geqslant 0.
\end{equation}

In order to show the inequality (\ref{decay zregularyzowanego}) for $p\in(1,\infty),$ we multiply equation
(\ref{+}) by $(\ue_x)^{p-1},$ and 
integrate the resulting equation over $\mathbb R$ to obtain
\begin{equation}
\label{!!}
\begin{split}
 \frac{1}{p}\frac{d}{dt} \int (\ue_x)^p dx&=\varepsilon\int \ue_{xx}(\ue_x)^{p-1}dx+\int (\ue_x)^{p-1}{\cal L} \ue_x dx\\
&-\int \left((\ue_x)^2+\ue \ue_x\right)(\ue_x)^{p-1} dx.
\end{split}
\end{equation}
The first integral on the right-hand side of (\ref{!!}) 
is equal to $\frac{\varepsilon}{p}\int \left[(\ue_x)^p\right]_xdx.$ Thus, since $\ue_x\in L^1(\mathbb R),$ this term equals zero.
The second integral on the right-hand side of (\ref{!!}) is non-positive 
by inequalities (\ref{convex inequality}) and (\ref{zero}) as well as by the assumptions on the kernel of the operator ${\cal L}$ from (\ref{as J}).
Thus, since $\ue_x$ is integrable and nonnegative, after the following calculations involving the integration by part on the third integral of the right-hand side of (\ref{!!})
\begin{equation*}
\int \left((\ue_x)^2+\ue \ue_x\right)(\ue_x)^{p-1} dx=\int (\ue_x)^{p+1} dx+\int \ue \left(\frac{(\ue_x)^p}{p}\right)_x dx=\left(1-\frac{1}{p}\right)\int(\ue_x)^{p+1}
\end{equation*}
we arrive at inequality
\begin{equation}
\label{A}
 \frac{1}{p}\frac{d}{dt} \Vert \ue_x(t) \Vert_p^p\leqslant -\left(1-\frac{1}{p}\right)\Vert \ue_x(t) \Vert_{p+1}^{p+1}.
\end{equation}
Combining inequality (\ref{A}) with the interpolation inequality 
\begin{equation*}
 \Vert \ue_x(t) \Vert_p^\frac{p^2}{p-1}\leqslant \Vert \ue_x(t) \Vert_{p+1}^{p+1}\Vert \ue_x(t) \Vert_1^\frac{1}{p-1}
\end{equation*}
and with the conservation of the $L^1$-norm in (\ref{-}) we obtain
the following differential inequality 
\begin{equation}
\label{nierownosc rozniczkowa}
 \frac{d}{dt} \Vert \ue_x(t) \Vert_p^p\leqslant -(p-1)\left(  \Vert \ue_x(t) \Vert_p^p\right)^\frac{p}{p-1}  \Vert u_{0,x}(t) \Vert_1^{-\frac{1}{p-1}}.
\end{equation}
Consequently, decay estimates (\ref{decay zregularyzowanego}) result from inequality (\ref{nierownosc rozniczkowa}) by standard calculations.

We obtain immediately the case of $p=\infty$ in inequality (\ref{decay zregularyzowanego}) by passing to the limit $p\to\infty.$

\noindent {\it Part II. Convergence towards rarefaction wave.} 

First, we recall that by Lemma \ref{w_R} the large time
asymptotics of $w(t)$ is described in $L^p(\R)$ by the rarefaction wave $w^R(t)$
and the rate of this convergence is $t^{-1/2(1-1/p)}$. 
Thus, it is enough to estimate $L^p$-norm of the difference of the solution $\ue$ of problem (\ref{rownanie zregularyzowane})-(\ref{warunek ue})
and of the smooth approximation of the rarefaction wave satisfying (\ref{eq-app})-(\ref{ini-app}).
To this end, using the following Gagliardo-Nirenberg-Sobolev inequality
\begin{equation*}
\label{GN}
\|v\|_p\leq C\|v_x\|^a_\infty\|v\|^{1-a}_1,
\end{equation*}
valid for every $1<p\leq \infty$ and for $a=1/2(1-1/p),$
inequality (\ref{decay zregularyzowanego}), and Lemma \ref{w_R} 
we have
\begin{align*}
\| \ue(t)-w(t)\|_p&
\leq C(\| \ue_x(t)\|_\infty+\|w_x(t)\|_\infty)^a\|
\ue(t)-w(t)\|_1^{1-a}
\\
&\leq Ct^{-a}\| \ue(t)-w(t)\|^{1-a}_1.
\end{align*}
Finally, the logaritmic estimate of the $L^1$-norm from Lemma \ref{log} completes the proof.
\endProof


\section{Passage to the limit $\varepsilon\to 0$}

\setcounter{equation}{0}
Here, we prove a result on the convergence as  $\varepsilon\to 0$ of solutions $\ue$ for regularized problem 
(\ref{rownanie zregularyzowane})-(\ref{warunek ue}) towards a weak solution to problem (\ref{rownanie})-(\ref{warunek}).

\begin{theorem}
 \label{zALS}
Let the assumptions on the initial data $u_0$ and the kernel $J$ from (\ref{as J})-(\ref{as u02}) hold true and let $\ue=\ue(x,t)$ be a solution to problem (\ref{rownanie zregularyzowane})-(\ref{warunek ue}) with $\varepsilon>0.$
Then, there exists a sequence $\varepsilon_n\to 0$ such that $u^{\varepsilon_n}\to u$ in $C([t_1,t_2],L^1_{loc}(\mathbb R)),$ 
for every $t_2>t_1>0,$ as well as  $u^{\varepsilon_n}\to u$ {\it a.e.} in $\mathbb R\times (0,\infty),$ where $u$ is a weak solution of problem (\ref{rownanie})-(\ref{warunek}).
\end{theorem}

In the proof of this theorem, the following version of the Aubin-Lions-Simon compactness theorem will be used.

\begin{theorem}
 \label{ALS}
Let $T > 0, 1 < p \leqslant\infty,$ and $1 \leqslant q \leqslant\infty.$ 
Assume that $Y\subset X\subset Z$ are Banach spaces such that $Y$ is compactly embedded in $X$
and $X$ is continuously embedded in $Z$. If $A$ is a bounded subset of $W^{1,p}([0,T],Z)$ and
of $L^q([0,T],Y),$ then $A$ is relatively compact  in $L^q([0,T],X)$
and in  $C([0,T],X)$ if $q=\infty$.
\end{theorem}

The proof of Theorem \ref{ALS} can be found in \cite{Si}.

\bigskip
\noindent
{\it Proof of Theorem \ref{zALS}.} First, we show the relative compactness 
of the family $\mathcal{F}=\{ \ue : \varepsilon\in(0,1]\}$ in the space $C((0,+\infty),L^1_{loc}(\mathbb R)),$ 
and next, we pass to the limit $\varepsilon\to 0,$ using 
the Lebesgue dominated convergence theorem. 

\noindent {\it Step} 1. 
We check the assumptions of the Aubin-Lions-Simon theorem in the case $p=q=\infty,$ $Y=W^{1,1}(K),$ $A={\bf 1}_{K\times [t_1,t_2]}\cal{F},$
$X=L^1(K)$ and $Z=(C^2_K)^*,$ with arbitrary $t_2>t_1>0,$ where $K\subset\mathbb R$ is a compact set and $(C^2_K)^*$ is topological dual space to the space of
$C^2$ functions with compact support in $K$ (with its natural norm).
First, we notice that $L^1(K)$ is obviously continuously embedded in $(C^2_K)^*,$ and by the Rellich-Kondrachov theorem 
$W^{1,1}(K)$ is compactly imbedded in $L^1(K).$ 
By inequality (\ref{L}), we have
\begin{equation*}
\left\vert \int_K \ue(t)\vp dx\right\vert\leqslant\Vert\vp\Vert_{C^2_K}\Vert u_0\Vert_\infty\vert K\vert
\end{equation*}
for every $\vp\in C^\infty_c (\mathbb R).$ Hence, the family $\cal{F}$ is bounded in $L^\infty([t_1,t_2],(C^2_K)^*).$
Now, we check that $\{\ue_t\}$ is bounded in 
$L^\infty([t_1,t_2],(C^2_K)^*).$ To this end, we multiply equation (\ref{rownanie zregularyzowane})
by $\vp\in C^\infty_c (\mathbb R)$ and integrate over $\mathbb R.$ Applying integrating by part formula we have the following estimate
\begin{equation}
\label{,}
 \left\vert \int_K \ue_t(t)\vp dx\right \vert\leqslant \Vert\vp\Vert_{C^2_K}\left(\varepsilon\int_K \vert\ue\vert dx+\int_K \vert{\cal L} \ue\vert dx+\int_K {(\ue)}^2dx\right).
\end{equation}
By assumption imposed on the kernel $J$ in (\ref{as J}), the Young inequality, and inequality (\ref{L}),
the right-hand side of inequality (\ref{,}) can be estimated by $\Vert\vp\Vert_{C^2_K}\Vert u_0\Vert_\infty|K|\left(\Vert J\Vert_1+1+\varepsilon+\frac{1}{2}\Vert u_0\Vert_\infty\right).$
\newline\indent Now, again  by inequality (\ref{L}) we have
$\int_K\vert\ue(t)\vert dx\leqslant\Vert u_0\Vert_\infty\vert K\vert.$ Moreover, from decay estimate (\ref{decay zregularyzowanego}) for $p=\infty$ we obtain
$\int_K\vert\ue_x(t)\vert dx\leqslant\frac{1}{t_1}\vert K\vert.$ All these estimates imply
that $\cal{F}$ is bounded in $L^\infty([t_1,t_2],W^{1,1}(K)).$ Thus, the Aubin-Lions-Simon theorem ensures that 
$\cal{F}$ is relatively compact in $C([t_1,t_2],L^1(K))$ for all $t_2>t_1>0,$ and all compact sets $K\subset\mathbb R.$

\noindent{\it Step} 2. 
We deduce from Step 1 and from the Cantor diagonal argument that there exists a sequence $\varepsilon_n\to 0$ and a function 
$u\in C((0,+\infty),L^1_{loc}(\mathbb R))$ such that $u^{\varepsilon_n}$ converges as $\varepsilon_n\to 0$ towards  
in $C([t_1,t_2 ],L^1(K))$ for all $t_2>t_1>0,$ and all compact $K\subset\mathbb R.$ 
Up to another subsequence, we can also assume that $u^{\varepsilon_n}\to u$ {\it a.e.} on $\mathbb R\times (0,\infty).$  
This convergence and inequality (\ref{L})
implies that $u\in L^\infty(\mathbb R\times (0,+\infty)).$
\newline\indent Now, we prove that a function $u$ is a weak solution of the problem (\ref{rownanie})-(\ref{warunek}).
To this end, we multiply equation (\ref{rownanie}) by $\vp\in C^\infty_c(\mathbb R\times [0,\infty)),$ 
and integrating the resulting equation over $\mathbb R\times [0,\infty)$ and integrating by parts, we obtain
\begin{equation}
\begin{split}
\label{ue_slabe_2}
&-\int_\mathbb R\int_0^\infty u^{\varepsilon_n}\vp_t dtdx-\int_\mathbb R u_0(x)\vp(x,0)~dx= \\
&=\varepsilon \int_\mathbb R\int_0^\infty u^{\varepsilon_n}\vp_{xx}~ dtdx+\int_\mathbb R\int_0^\infty u^{\varepsilon_n} {\cal L}\vp ~dtdx+\frac{1}{2}\int_\mathbb R\int_0^\infty (u^{\varepsilon_n})^2\vp_x~dtdx.
\end{split}
\end{equation}
Thus, since $u^{\varepsilon_n}\to u$ {\it a.e.} as $\varepsilon_n\to 0,$ the sequence $\{u^{\varepsilon_n}\}$ is bounded in $L^\infty$-norm 
by $\Vert u_0\Vert_\infty,$ and ${\cal L}\vp$
is integrable, the Lebesgue dominated convergence theorem allows us to pass to the
limit in equality (\ref{ue_slabe_2}). This completes the proof of Theorem \ref{zALS}.
\endProof

\vskip 1 cm
Now, we are in a position to prove Theorem \ref{th main}.

\bigskip
\noindent
{\it Proof of Theorem \ref{th main}.}
Denote by $u^{\varepsilon_n}$ the solution of regularized problem (\ref{rownanie zregularyzowane})-(\ref{warunek ue}) and by $u$ the weak solution 
of problem (\ref{rownanie})-(\ref{warunek}).
By Theorem \ref{zALS}, we know that $u^{\varepsilon_n}\to u$ {\it a.e.} on $\mathbb R\times (0,\infty)$ for a sequence $\varepsilon_n\to 0.$
Therefore, by the Fatou lemma and Theorem~\ref{jeszcze nie wiem jaki label}, we have for each 
$R>0$ and $p\in[1,\infty]$ and for all $t>0$ the following estimate 
\begin{equation*}
\Vert u(t)-w^R(t)\Vert_{L^p(-R,R)}\leqslant \liminf_{\varepsilon_n\to 0} \Vert u^{\varepsilon_n}(t)-w^R(t)\Vert_{L^p(-R,R)}\leqslant C t^{-(1-1/p)/2}[\log(2+t)]^{(1+1/p)/2}.
\end{equation*}
Since $R > 0$ is arbitrary and the right-hand side of this inequality does not depend
on $R$, we complete the proof of the inequality (\ref{glowna nierownosc}) by letting $R\to \infty.$

Since solution of the regularized problem satisfy the $L^1$-contraction property stated in Theorem \ref{wlasnosci rozwiazania}, by an analogous 
passage to the limit $\varepsilon_n\to 0$ as described above, we obtain $L^1$- contraction inequality for weak solutions to the nonlocal 
problem (\ref{rownanie})-(\ref{warunek}). Hence a weak solution to (\ref{rownanie})-(\ref{warunek}) is unique.
\endProof


\bigskip

{\bf Acknowlegements.} This work was supported by the MNiSW grant No. IdP2011/000661.


\end{document}